\documentclass[11pt,reqno]{amsart}
\usepackage[body={7in,9in},centering]{geometry}
\usepackage{amsmath,amssymb,amsthm,mathrsfs,extarrows}
\usepackage{color,xcolor}
\definecolor{cobalt}{RGB}{61,99,181}
\usepackage[colorlinks,citecolor=cobalt,linkcolor=cobalt]{hyperref}
\usepackage{float}
\usepackage{exscale}
\usepackage{relsize}
\usepackage{graphicx}
\usepackage{tikz,tikz-cd}

\date{\today}
\newtheorem{thm}{Theorem}[section]
\newtheorem{defi}[thm]{Definition}
\newtheorem{cor}[thm]{Corollary}
\newtheorem{lem}[thm]{Lemma}
\newtheorem{rem}[thm]{Remark}
\newtheorem{ex}[thm]{Example}
\newtheorem{prop}[thm]{Proposition}

\numberwithin{equation}{section}
\makeatletter

\newcommand{\Rmnum}[1]{\expandafter\@slowromancap\romannumeral #1@}
\makeatother

\newcommand{\N}{\mathbb{N}}

\begin{document}

	\title[disjoint Ces$\grave{A}$ro-hypercyclic operators]{disjoint Ces$\grave{A}$ro-hypercyclic operators}
	
	\author[Qing Wang And Yonglu Shu]{Qing Wang And Yonglu Shu}

	
%	\subjclass[2010]{47A30, 47B37, 05C05}

	\begin{abstract}
	In this paper, we investigate the properties of disjoint Ces$\grave{a}$ro-hypercyclic \mbox{operators}. First, the definition of disjoint Ces$\grave{a}$ro-hypercyclic \mbox{operators} is provided, and disjoint Ces$\grave{a}$ro-Hypercyclicity Criterion is proposed. Later, two methods are used to prove that operators satisfying this criterion possess disjoint Ces$\grave{a}$ro-hypercyclicity. Finally, this \mbox{paper} further investigates weighted shift operators and provides detailed characterizations of the weight sequences for disjoint Ces$\grave{a}$ro-hypercyclic unilateral and bilateral weighted shift operators on sequence spaces.
		%In this paper, to begin with, we introduce the concept of disjoint Ces$\grave{a}$ro-hypercyclic of finite operators in Banach space; next, give a disjoint Ces$\grave{a}$ro-Hypercyclicity Criterion and study the disjointness in Ces$\grave{a}$ro-hypercyclic; finally, characterize the weight sequences of disjoint Ces$\grave{a}$ro-hypercyclic weighted shift operators.
			\end{abstract} \maketitle
		 %We give a criterion to study the disjoint dynamical properties of Ces$\grave{a}$ro-hypercyclic operators. Two methods are used to prove that if Ces$\grave{a}$ro-hypercyclic operators satisfy this criterion, they are disjoint Ces$\grave{a}$ro-hypercyclic.
	 %\keywords{disjoint Ces$\grave{a}$ro-hypercyclic operator; disjoint Ces$\grave{a}$ro-hypercyclic vector; Banach spaces}
	\section{Introduction}\label{Intro}
	Let $X$ be a separable infinite dimensional complex Banach space, and $B(X)$ be the space of all bounded linear operators on $X.$ Let $\mathbb{N}$ and $\mathbb{Z}$ denote positive integers and integers, and $\mathbb{K}$ denote the real scalar field $\mathbb{R}$ and complex scalar field $\mathbb{C}$. A bounded linear operator $T:X\rightarrow X$ is called hypercyclic (supercyclic) if there is some $x\in X$ such that the orbit $Orb(T,x)=\{x,Tx,T^2x,\ldots\}$ ($Orb\{\lambda T^nx:\lambda\in\mathbb{K},n=0,1,2,\ldots\}$)is dense in $X$, and $x$ is called a hypercyclic (supercyclic) vector of $T$. The set of hypercyclic (supercyclic) vectors of $T$ is denoted by $HC(T)$ ($SC(T)$). The first example of hypercyclic operators on Banach space was provided by Rolewicz \cite{SR}, he showed that if $B$ is a unilateral backward shift on Banach space $\ell^p(\mathbb{N})\;(1\leqslant p<\infty)$, then $\lambda B$ is hypercyclic if and only if $|\lambda| >1$. Kitai \cite{KC} in 1984, Gethner and Shapiro \cite{SC} in 1987 independently gave a Hypercyclicity Criterion, promoting the study of linear operator dynamics. But not every hypercyclic operator satisfies Hypercyclicity Criterion, De La Rosa and Read \cite{DRR} constructed a Banach space that supports a hypercyclic operator which is not weakly mixing, and thus fails to satisfy the Hypercyclicity Criterion. More about hypercyclic operators and linear operator dynamics can be seen in \cite{GEM,BF}. %[\cite{BPA},\cite{GS},\cite{SHN},\cite{BG},\cite{ASI}],[\cite{GEM},\cite{BF}]. %But the study of hypercyclc operators did not go vary far until 1984, Katai \cite{KC} discussed the problem of invariant subsets of linear operators and gave a sufficient condition on the hypercyclicity of operators, the study of linear operator dynamics has been deepened. But not every hypercyclic operator satisfies this condition, De La Rosa, M and Read, C \cite{DRR} constructed a Banach space that supports a hypercyclic operator which is not weakly, and thus fails to satisfy the Hypercyclicity Criterion. More about hypercyclic operators and linear operator dynamics can be seen in [\cite{BPA},\cite{GS},\cite{SHN},\cite{BG},\cite{ASI}],[\cite{GEM},\cite{BF}].
	
	On the basis of the study of hypercyclic operators, Bernal-González \cite{BGL}, Bès and Peris \cite{BP} in 2007 independently gave the concept of disjoint hypercyclic operators. For $(N\geqslant 2)$ hypercyclic operators $T_1,\ldots,T_N$ acting on a separable infinite dimensional Banach space $X$ are disjoint hypercyclic (in short, d-hypercyclic) if there is some vector $x\in X$ such that the orbit
	$$\{(x,\ldots,x),(T_1x,\ldots,T_Nx),(T_1^{2}x,\ldots,T_N^2x),\ldots,(T_1^n{x},\ldots,T_N^nx),\ldots\}$$
	is dense in $X^N$, and $x$ is called a d-hypercyclic vector of $T_1,\ldots,T_N$, the set of d-hypercyclic vectors is denoted by $d-HC(T_1,\ldots,T_N)$. If the set of d-hypercyclic vectors of $T_1,\ldots,T_N$ is a dense set in $X$, then $T_1,\ldots,T_N$ are densely d-hypercyclic.  Shkarin \cite{SS} gave a proof of existence of disjoint hypercyclic operators. Bès and Peris \cite{BP} gave a criterion to prove d-hypercyclicity, called d-Hypercyclicity Criterion, which expands the Hypercyclicity Criterion. %More about disjoint hypercyclic operators can be seen in [\cite{SRSS},\cite{MO},\cite{BMS}].
	%\begin{rem}
	%	Hypercyclic and d-hypercyclic have a great difference.
	%\end{rem}
	%Bès, J. and Peris, A \cite{BP} show that a continuous linear operator on a Banach space satisfies the Hypercyclicity Criterion if and only if 
	
	Let $M_n(T)$ denote the arithmetic mean of the powers of $T\in B(X)$, that is 
	$$M_n(T)=\frac{1}{n}(I+T+T^2+\ldots+T^{n-1}),n\in \mathbb{N}.$$
	If the arithmetic mean of the orbit of $x$ is dense in $X$, then $T$ is said to be Ces$\grave{a}$ro-hypercyclic. León-Saavedra \cite{LS} proved that an operator is Ces$\grave{a}$ro-hypercyclic if and only if there exists a vector $x\in X$ such that the orbit $\{n^{-1}T^nx\}_{n\geqslant 1}$ is dense in $X$. El Berrag and Tajmouati \cite{ET} proposed a Ces$\grave{a}$ro-Hypercyclicity Criterion on Hilbert space, which has important applications in proving the Ces$\grave{a}$ro-hypercyclicity of operators. In this paper when we refer to the Ces$\grave{a}$ro orbit of a vector $x$ under an operator $T$, we mean the sequence $\{n^{-1}T^nx:n=1,2,\ldots\}.$  
	
	Weighted shifts are one of the important research subjects in the dynamics of linear operators. León-Saavedra \cite{LS} showed that every Ces$\grave{a}$ro-hypercyclic unilateral weighted shift is hypercyclic on $\ell^2(\mathbb{N})$, while for bilateral weighted shifts on $\ell^2(\mathbb{Z})$ are not the case: there are operators that are not hypercyclic, but their Ces$\grave{a}$ro means are hypercyclic. Tajmouati and El \cite{TAEB} gave an example, which is hypercyclic and supercyclic, but it is not Ces$\grave{a}$ro-hypercyclic. 
	%More about Ces$\grave{a}$ro-hypercyclic operators can be seen in
	%[\cite{EM},\cite{CH},\cite{CCC}]. %[\cite{EM},\cite{CH},\cite{CCC},\cite{CG},\cite{}].
	
		%\begin{prop}
		%Let $T:l^2(\mathbb{Z})\rightarrow l^2(\mathbb{Z})$ be a bilateral weighted shift with weighte sequence $\{w_n\}_n\in \mathbb{Z}$, Then $T$ is ces$\grave{a}ro-hypercyclic$ if and only if there exists an increasing sequence $\{n_k\}$ of positive integers such that for any integer $q$,
		%\begin{equation*}
			
		%\lim_{k\rightarrow\infty}\prod_{i=1}^{n_q}\frac{w_{i+q}}{n_k}=\infty\;\text{and}\;\lim_{k\rightarrow\infty}\prod_{i=0}^{n_k-1}\frac{w_{q-i}}{n_k}=\infty
%	\end{equation*}
		%\end{prop}
	\begin{ex}\cite{LS}
		Let $B\in \ell^2(\mathbb{Z})$ be the bilateral backward shift with the weight sequence,
		\begin{equation*}
			w_n=
			\begin{cases}
				1\quad&\text{if}\;n\leqslant 0,\\
				2\quad&\text{if}\;n\geqslant 1.
			\end{cases}
		\end{equation*}
	Then $B$ is Ces$\grave{a}$ro-hypercyclic, but it is not hypercyclic.
		
	\end{ex}
	\begin{ex}\cite{TAEB}
		Let $B\in \ell^2(\mathbb{Z})$  be the bilateral backward shift with the weight sequence,
		\begin{equation*}
			w_n=
			\begin{cases}
				2\quad&\text{if}\;n< 0,\\
				\frac{1}{2}\quad&\text{if}\;n\geqslant 0.
			\end{cases}
		\end{equation*}
		Then $B$ is hypercyclic and supercyclic, but it is not Ces$\grave{a}$ro-hypercyclic.
		
	\end{ex}
	\begin{ex} \cite{TAEB}
		Let $B\in \ell^2(\mathbb{Z})$ be the bilateral backward shift with the weight sequence,
			\begin{equation*}
			w_n=
			\begin{cases}
				\frac{1}{2}\quad&\text{if}\;n< 0,\\
				n+1\quad&\text{if}\;n\geqslant 0.
			\end{cases}
		\end{equation*}
		Then $B$ is Ces$\grave{a}$ro-hypercyclic, but it is not hypercyclic and supercyclic.
		\end{ex}
	 From the hypercyclicity and disjoint hypercyclicity, we generalize Ces$\grave{a}$ro-hypercyclic to disjoint Ces$\grave{a}$ro-hypercyclic. In this process of generalization, we study the dynamical properties of disjoint Ces$\grave{a}$ro-hypercyclic operators.

	In the second part, we introduce the basic definitions, including disjoint Ces$\grave{a}$ro-hypercyclic, disjoint Ces$\grave{a}$ro-topologically transitive, and disjoint Ces$\grave{a}$ro-mixing.
	
	In the third part, we discuss the properties of disjointness in Ces$\grave{a}$ro-hypercyclic operators. Meanwhile, we give disjoint Ces$\grave{a}$ro-Hypercyclicity Criterion and use two methods to prove that the Ces$\grave{a}$ro-hypercyclic operators satisfying this criterion are disjoint Ces$\grave{a}$ro-hypercyclic. On the one hand, we construct a disjoint Ces$\grave{a}$ro-hypercyclic vector by disjoint Ces$\grave{a}$ro-Hypercyclicity Criterion. On the other hand, we prove that operators satisfying disjoint Ces$\grave{a}$ro-Hypercyclicity Criterion are disjoint Ces$\grave{a}$ro-Blow-Up/Collapse, then they are disjoint Ces$\grave{a}$ro-hypercyclic.%Through these two methods, we can know that the Ces$\grave{a}$ro operators satisfying this criterion are disjoint Ces$\grave{a}$ro-hypercyclic operators. By applying this theorem, we can prove which operators are disjoint Ces$\grave{a}$ro-hypercyclic.
	
	In the last part, as an application of the results from the third part, we characterize the weight sequences of disjoint Ces$\grave{a}$ro-hypercyclic weighted shifts.
	
	\section{disjoint ces$\grave{A}$ro-hypercyclic}\label{S2}
	%In this section, we recall the definition of Hardy spaces and some known results which will be needed in the next section.
	\begin{defi}\cite{LS}
		We say that an operator $T$ acting on a separable infinite dimensional complex Banach space $X$ is Ces$\grave{a}$ro-hypercyclic if there exists a vector $x\in X$ such that the orbit $$\Big\{x,Tx,\frac{T^2x}{2},\frac{T^3x}{3},\ldots,\frac{T^nx}{n},\ldots\Big\}$$
		is dense in $X$, such $x$ is called a Ces$\grave{a}$ro-hypercyclic vector of $T$, the set of Ces$\grave{a}$ro-hypercyclic vectors of $T$ is denoted by $CH(T)$.
		
	\end{defi}
%	\begin{defi}
	%	Let $X$ be a infinite dimensional Banach space and $T$ be a bounded linear opweator. If there exists a vector $x\in X$ such that the orbit $$\{x,Tx,\frac{T^2x}{2},\frac{T^3x}{3},\ldots,\frac{T^nx}{n},\ldots\}$$
	%	is dense in $X$, then $T$ is Ces$\grave{a}$ro-hypercyclic. Such $x$ is called a Ces$\grave{a}$ro-hypercyclic vector, the set of Ces$\grave{a}$ro-hypercyclic vector of $T$ is denoted by $CH(T)$.
%	\end{defi}
	
	%定义的等价性还需要重新说明下。

	\begin{defi}\label{Toeplitz}
		We say that $(N\geqslant 2)$ Ces$\grave{a}$ro-hypercyclic operators $T_1,\ldots,T_N$ acting on a separable infinite dimensional complex Banach space $X$ are disjoint Ces$\grave{a}$ro-hypercyclic (in short, d-Ces$\grave{a}$ro-hypercyclic ) if there is some vector $x\in X$ such that the orbit\\
		$$\Big\{(x,\ldots,x),(T_1x,\ldots,T_Nx),(\frac{T_1^2x}{2},\ldots,\frac{T_N^2x}{2}),\ldots,(\frac{T_1^nx}{n},\ldots,\frac{T_N^nx}{n}),\ldots\Big\}$$
		is dense in $X^N$.
		
		In particular, we say that $x$ is a d-Ces$\grave{a}$ro-hypercyclic vector of $T_1,T_2,\ldots,T_N$. The set of d-Ces$\grave{a}$ro-hypercyclic vectors of $T_1,T_2,\ldots,T_N$ is denoted by $d-CH(T_1,T_2,\ldots,T_N)$. If the set of d-Ces$\grave{a}$ro-hypercyclic vectors of $T_1,\ldots,T_N$ is a dense set in $X$, then $T_1,\ldots,T_N$ are densely d-Ces$\grave{a}$ro-hypercyclic. 
		
	\end{defi}
	
	Accordingly, we give the concepts of disjoint Ces$\grave{a}$ro-topologically transitive and disjoint Ces$\grave{a}$ro-mixing.
	
	\begin{defi}\label{inv}
		We say that $T_1,T_2,\ldots,T_N\in B(X)\,(N\geqslant 2)$ are disjoint Ces$\grave{a}$ro-topologically transitive if for every N+1 non-empty open subsets $U,V_1,\ldots,V_N$ of $X$, there exists $n\geqslant 1$ such that $U\cap(nT_1^{-n})(V_1)\cap\cdots\cap(nT_N^{-n})(V_N)\neq\varnothing.$
	\end{defi}
	\begin{defi}\label{irr}
		We say that $T_1,T_2,\ldots,T_N\in B(X)\,(N\geqslant 2)$ are disjoint Ces$\grave{a}$ro-mixing if for every N+1 non-empty open subsets $U,V_1,\ldots,V_N$ of $X$, there exists $m\geqslant 1$ such that $U\cap(nT_1^{-n})(V_1)\cap\cdots\cap(nT_N^{-n})(V_N)\neq\varnothing$ for each $n\geqslant m$.
	\end{defi}
	
	For convenience, we use d-Ces$\grave{a}$ro-hypercyclic instead of disjoint Ces$\grave{a}$ro-hypercyclic.
	%把下面所有的disjoint 换成d-。
	%
		\section{A criterion for d-ces$\grave{A}$ro-hypercyclic}\label{s3}
	Based on the articles by Bès, Peris \cite{BP} and El Berrag \cite{EM}, we have:
	\begin{prop}\label{6.2}
		Let $T_1,T_2,\ldots,T_N\in B(X)\;(N\geqslant 2)$.
		The following are equivalent:\\
		(1) The set of $d-CH(T_1,T_2,\ldots,T_N)$ is a dense $G_\delta$ set.\\
		(2) $T_1,T_2,\ldots,T_N$ are d-Ces$\grave{a}$ro-topologically transitive.\\
		(3) For each $x,y_1,\ldots,y_N\in X$, there exist sequence $(x_k)\subseteq X$ and $(n_k)\subseteq \mathbb{N}$ such that
		$$x_k\rightarrow  x\ \text{and}\
		\frac{1}{n_k}T_l^{n_k}{x_k}\rightarrow y_l \ \text{as}\	k\rightarrow\infty.$$
		(4) For each $x,y_1,\ldots,y_N\in X$, and each neighbourhood $W$ of the zero in $X$, there exist $z\in X$ and $n\geqslant1$ such that
		$$z-x\in W\ \text{and}\
		\frac{1}{n}T_l^{n}{z}-y_l\in W\,(1\leqslant l\leqslant N).$$
	\end{prop}
	
	\begin{proof}
		(1)$\Rightarrow$(2). Since the set of
		$d-CH(T_1,T_2,\ldots,T_N)$ is a dense $G_\delta$ set, for any $N+1$ nonempty open subsets $U,V_1,\ldots,V_N$ of $X$, we can find $x_0\in U\cap d-CH(T_1,T_2,\ldots,T_N)$ and $m\in \mathbb{N} $ such that $\frac{1}{m}T_i^m(x_0)\in V_i\;(1\leqslant i\leqslant N)$, then $U\cap (mT_1^{-m})V_1\cap\cdots \cap (mT_N^{-m})V_N\neq \varnothing$, $T_1,T_2,\ldots,T_N$ are d-Ces$\grave{a}$ro-topologically transitive.

		(2)$\Rightarrow$(1). Let$\{A_j,j\in \mathbb{N}\}$ be a basis for the topology of $X$. We have $x\in d-CH(T_1,T_2,\ldots,T_N)$ if and only if the orbit $\big\{(\frac{T_1^nx}{n},\ldots,\frac{T_N^nx}{n}):n\geqslant 1\big\}$ is dense in $X^N$ if and only if for each $J=(j_1,\ldots,j_N)\in \mathbb{N}^N$, there exists $n\geqslant 1$ such that $\frac{1}{n}T_i^nx\in A_{j_i}\;(1\leqslant i\leqslant N)$ if and only if $x\in\cap_{J\in\mathbb{N}^N}\cup_{n\geqslant 1}\big((nT_1^{-n})A_{j_1}\cap\cdots\cap(nT_N^{-n})A_{j_N}\big).$ By (2), for every non-empty open subset $U$ of $X$ and for all $J\in\mathbb{N}^N$, there exists $n\geqslant 1$ such that $U\cap (nT_1^{-n})A_{j_1}\cap\cdots\cap(nT_N^{-n})A_{j_N}$ is non-empty and open. The set of 
		$$\cup_{n\geqslant 1}\left((nT_1^{-n})A_{j_1}\cap\cdots\cap(nT_N^{-n})A_{j_N} \right)$$ is denoted by $C_J$ and $C_J$
		is non-empty and open. Furthermore, $U\cap C_J\neq\varnothing$ for all $J\in\mathbb{N}^N$. Thus $C_J$ is dense in $X$ and by the Baire category theorem, the set of $d-CH(T_1,T_2,\ldots,T_N)$ is $\cap_{J\in\mathbb{N}^N}\cup_{n\geqslant 1}\left((nT_1^{-n})A_{j_1}\cap\cdots\cap(nT_N^{-n})A_{j_N} \right)$, which is a dense $G_\delta$ set.
		 
		 %$B(x,1)\supset B(x,\frac{1}{2})\supset B(x,\frac{1}{3})\supset\ldots$ and
		 %$B(y_i,1)\subset B(y_i,\frac{1}{2})\subset B(y_i,\frac{1}{3})\subset\ldots$
		(2)$\Rightarrow$(3). Let $x,y_1,\ldots,y_N\in X$. For all $k\geqslant 1$, $B(x,\frac{1}{k})$ and $B(y_i,\frac{1}{k})\;(1\leqslant i\leqslant N)$ are non-empty open subsets. Since $T_1,T_2,\ldots,T_N$ are d-Ces$\grave{a}$ro-topologically transitive, then there exist $(n_k)\in \mathbb{N}$ and $(x_k)\in X$ such that $x_k\in B(x,\frac{1}{k})$ and $\frac{1}{n_k}{T_l^{n_k}x_k}\in B(y_l,\frac{1}{k})$ for all $k\geqslant 1$. Then we have $||x_k-x||<\frac{1}{k}$ and $||\frac{1}{n_k}{T_l^{n_k}x_k}-y_l||<\frac{1}{k}$ for all $k\geqslant 1$.
		%(2)$\Rightarrow$(3): Let $x,y_1,\ldots,y_N\in X$, $B(x,\frac{1}{n})=\cap_{m=1}^{n}B(x,\frac{1}{m})$, and $B(y_i,\frac{1}{n})=\cap_{m=1}^{n}B(y_i,\frac{1}{m})\;(1\leqslant i\leqslant N)$. Since $T_1,T_2,\ldots,T_N$ are d-Ces$\grave{a}$ro-topologically transitive, then there exist $(n_k)\in \mathbb{N}$ and $(x_k)\in X$ such that $x_k\in B(x,\frac{1}{k})$ and $\frac{1}{n_k}{T_l^{n_k}x_k}\in B(y_i,\frac{1}{k})$ for all $k\geqslant 1$. Then we have $||x_k-x||<\frac{1}{k}$ and $||\frac{1}{n_k}{T_l^{n_k}x_k}-y_l||<\frac{1}{k}$ for all $k\geqslant 1$.
		
		(3)$\Rightarrow$(4). Obviuously.
		
		(4)$\Rightarrow$(2). Let $ U,V_1,\ldots,V_N $ be nonempty open subsets of $X$. Pick $x\in U$, $y_i\in V_i$ and a neighbourhood $W$ of zero such that $x+W\in U$ and $y_i+W\in V_i\;(1\leqslant i\leqslant N)$, by (4) there exist $z\in X$ and $n\geqslant1$ such that $z-x\in W$ and $\frac{1}{n}{T_i^n}z-y_i\in W$, then $z\in U$ and $\frac{1}{n}T_i^nz\in V_i\;(1\leqslant i\leqslant N)$, i.e., $U\cap (nT_1^{-n})V_1\cap\cdots \cap (nT_N^{-n})V_N\neq \varnothing$.
		%有些别扭
	\end{proof}
	Godefroy and Shapiro \cite{GS} gave a sufficient condition for hypercyclicity, which more vividly depicts the definition of transitivity. Bès and Peris \cite{BP} extended this condition to d-hypercyclicity. Inspired by these conclusions, we extend this condition to d-Ces$\grave{a}$ro-hypercyclicity.
	\begin{defi}\label{6.9}
		(d-Ces$\grave{a}$ro-Blow-Up/Collapse) Let $T_1,T_2,\ldots,T_N\in B(X)\,(N\geqslant 2)$, suppose that for each open neighbourhood $W$ of zero of $X$ and non-empty open subsets $V_0 ,V_1,\ldots,V_N$ of $X$, there exists $m\in\N$ such that
		$$W\cap(mT_1^{-m})(V_1)\cap\cdots\cap(mT_N^{-m})(V_N)\neq\varnothing,$$
		$$V_0\cap (mT_1^{-m})(W)\cap \cdots \cap(mT_N^{-m})(W)\neq \varnothing.$$
		Then $T_1,T_2,\ldots,T_N$ are d-Ces$\grave{a}$ro-Blow-Up/Collapse.
		% d-Ces$\grave{a}$ro-topologically transitive, in particular, they are d-Ces$\grave{a}$ro-hypercyclic.
		
	\end{defi}
	
	\begin{prop} \label{pr3.3}
		If $T_1,T_2,\ldots,T_N$ are d-Ces$\grave{a}$ro-Blow-Up/Collapse, then they are d-Ces$\grave{a}$ro-topologically transitive.
	\end{prop}
	\begin{proof}
		Let $ V_0,V_1,\ldots,V_N $ be non-empty open subsets of $X$, there are non-empty open subsets $ U_0,U_1,\ldots,U_N $ and a 0-neighbourhood $W$ such that $U_i+W\subset V_i\;(0\leqslant i\leqslant N)$. From the hypothesis, we have $m\in\mathbb{N}$, $\eta_0\in U_0$ and $w\in W$ such that $\frac {1}{m}T_l^m(\eta_0)\in W$ and $\frac {1}{m}T_l^m(w)\in U_l\;(1\leqslant l\leqslant N)$. We can know that $\eta_0+w\in U_0+W\subset V_0$ and $\frac{1}{m}T_l^m(\eta_0+w)=\frac{1}{m}T_l^m(\eta_0)+\frac{1}{m}T_l^m(w)\in W+U_l\subset V_l$, i.e., $V_0\cap mT_1^{-m}(V_1)\cap\cdots \cap mT_N^{-m}(V_N)\neq \varnothing.$ 
		Hence $T_1,\ldots,T_N$ are d-Ces$\grave{a}$ro-topologically transitive.
	\end{proof}
	Let's introduce an important criterion that will have a lot of application in our later content.
	\begin{defi}\label{7.11}
		(d-Ces$\grave{a}$ro-Hypercyclicity \mbox{Criterion}) We say that operators $T_1,\ldots,T_N\in B(X)(N\geqslant 2)$ satisfy the disjoint Ces$\grave{a}$ro-Hypercyclicity Criterion (in short, d-Ces$\grave{a}$ro-Hypercyclicity \mbox{Criterion}) with respect to an increasing sequence $(n_k)$ of positive intergers provided there exist dense subsets $X_0,X_1,\ldots,X_N$ of $X$ and mappings $S_{l,k} :X_l\rightarrow X(1\leqslant l\leqslant N,k\in{\mathbb{N}})$ satisfying
		
		(1)  $\frac{1}{n_k}T_l^{n_k}\rightarrow0\;\text{pointwise on}\; X_0\;\text{as}\;k\rightarrow\infty,$

		(2)
		$S_{l,k}\rightarrow 0 \;\text{pointwise on}\, X_l\;\text{as}\;k\rightarrow\infty, $

		(3)  $(\frac{1}{n_k}T_l^{n_k}S_{i,k}-\delta_{i,l}Id_{X_i})\rightarrow0 \;\text{pointwise on}\; X_i\;\text{as}\;k\rightarrow\infty.$
		
		% We say $T_1,\ldots,T_N$ satisfy the disjoint Ces$\grave{a}$ro-Hypercyclicity Criterion with respect to a sequence $(n_k)$ if they satisfy the above conditions.
		
	\end{defi}\label{7.11.1}
	%%	If $A\in\mathcal{I}(L^{\infty}(\mathbb{T}))$, then $$\lim\limits_{r\to1^-}\langle Ak_{re^{i\theta}},k_{re^{i\theta}}\rangle\ \  exists$$ for almost every $e^{i\theta}\in\mathbb{T}$, where $k_{z}(e^{i\theta})=\frac{\sqrt{1-|z|^2}}{1-ze^{-i\theta}}$ is the normalized reproducing kernel for $H^2$. In this case, we denote the above limits by $title{A}(e^{i\theta})$. Note that $|\widetilde{A}(e^{i\theta})|\leqslant\|A\|$, so $\tilde{A}\in L^{\infty}(\mathbb{T})$.
	%%\end{lem}
	%We can easily get Ces$\grave{a}$ro-Hypercyclicity Criterion whenever $N=1$.
	 %Let's prove that operators $T_1,\ldots,T_N$ satisfying the d-Ces$\grave{a}$ro-Hypercyclicity Criterion are d-Ces$\grave{a}$ro-hypercyclic.
	% In the following proof, we construct a d-Ces$\grave{a}$ro-hypercyclic vector.
	\begin{thm}\label{dchc}
	If operators $T_1,\ldots,T_N$ satisfy the d-Ces$\grave{a}$ro-Hypercyclicity Criterion, then $T_1,\ldots,T_N$ are d-Ces$\grave{a}$ro-hypercyclic.
	\end{thm}
	\begin{proof}
		Since $X$ is separable, then $X^N$ is separable. We can assume that $X_1\times\cdots\times X_N$ is a countable dense set of $X^N$, which is denoted by $\{(y_{1j},y_{2j},\ldots,y_{Nj}):y_{lj}\in X_l,1\leqslant l\leqslant N,j\in \mathbb{N}\}$.
		
		 Claim: there exist $x_j\in X$ and $k_j\in \mathbb{N}$ such that 
		$$x_1+{S_{1,k_1}}y_{11}+\cdots+{S_{N,k_1}}y_{N1}+x_2+{S_{1,k_2}}y_{12}+\cdots+{S_{N,k_2}}y_{N2}+x_3+\cdots$$
		convergents, we denote it by $x$ and $x$ is a d-Ces$\grave{a}$ro-hypercyclic vector.
		
		Next, we will construct $x_j$ and $k_j$ such that for $j\geqslant 1,1\leqslant m\leqslant N$, we have
		\begin{align}\label{eqm.1}%\textcircled{1}\label{1}
			\|x_j\|<\frac{1}{2^j},
	\end{align}
		\begin{align}\label{eqm.2}%\textcircled{2}\label{eq2}
		\bigg\|\frac{1}{n_{k_l}}{T_m^{n_{k_l}}}{x_j}\bigg\|<\frac{1}{2^j}\;(l=1,\ldots,j-1),
	\end{align}
	\begin{align}\label{eqm.3}%\textcircled{3}\label{3}
		\Big\|S_{i,k_j}y_{ij}\Big\|<\frac{1}{2^j}\;(1\leqslant i\leqslant N ),
	\end{align}
	\begin{align}\label{eqm.4}%\textcircled{4}\label{4}
	\bigg\|\frac{1}{n_{k_l}}{T_m^{n_{k_l}}}{S_{i,k_j}}{y_{ij}}\bigg\|<\frac{1}{2^j}\;(l=1,\ldots,j-1)\;(1\leqslant i\leqslant N ),%\,\textcircled{4}\label{4}
\end{align}
		%$$||\frac{{T_m^{n_{k_j}}}{S_{m,n_{k_j}}}{y_{mj}}}{n_{k_j}}-y_{mj}||<\frac{1}{2^j}\;||\frac{{T_m^{n_{k_j}}}(\sum_{l=1}^{j-1}(x_l+\sum_{i=1}^{N}{S_{i,n_{k_l}}}{y_{ml}})+x_j){n_{k_j}}||<\frac{1}{2^j}.$$
			\begin{align}\label{eqm.5}%\textcircled{5}\label{5}
			\bigg\|\frac{1}{n_{k_j}}T_m^{n_{k_j}}S_{m,k_j}y_{mj}-y_{mj}\bigg\|<\frac{1}{2^j},
		\end{align}
		\begin{align}\label{eqm.6}%\textcircled{6}\label{6}
			\bigg\|\frac{1}{n_{k_j}}T_m^{n_{k_j}}S_{i,k_j}y_{ij}\bigg\|<\frac{1}{2^j}\;(1\leqslant i\leqslant N,i\neq m),
	\end{align}
	\begin{align}\label{eqm.7}%\textcircled{7}\label{l}
			\bigg\|\frac{1}{n_{k_j}}T_m^{n_{k_j}}\big(\sum_{l=1}^{j-1}(x_l+\sum_{i=1}^{N}S_{i,k_l}y_{il}\big)+x_j)\bigg\|<\frac{1}{2^j}.
		\end{align}
		
			If $j=1$, (\ref{eqm.2})(\ref{eqm.4}) do not exist.
			 For (\ref{eqm.1})(\ref{eqm.7}), let $x_1=0$.
			 For (\ref{eqm.3}), by d-Ces$\grave{a}$ro-Hypercyclicity Criterion (2), we can know that there is $k_{v_1}$ so that when $k_1^*>k_{v_1}$, we have
			$$\Big\|S_{i,k_1^*}y_{i1}\Big\|<\frac{1}{2}.$$
			For (\ref{eqm.5})(\ref{eqm.6}), by d-Ces$\grave{a}$ro-Hypercyclicity Criterion (3), we can know that there is $k_{v_2}$ so that when $k_1^*>k_{v_2}$, we have
			$$\bigg\|\frac{1}{n_{k_1^*}}T_m^{n_{k_1^*}}S_{m,k_1^*}y_{m1}-y_{m1}\bigg\|<\frac{1}{2},\qquad\bigg\|\frac{1}{n_{k_1^*}}T_m^{n_{k_1^*}}S_{i,k_1^*}y_{i1}\bigg\|<\frac{1}{2}\;(1\leqslant i\leqslant N,i\neq m).$$
				Taking $k_1=\max\{k_{v_1}+1,k_{v_2}+1\}$, then (\ref{eqm.1})-(\ref{eqm.7}) hold.
				
			If $j\geqslant 2$, we assume that $x_1,\ldots,x_{j-1}$ and $k_1,\ldots,k_{j-1}$ have been found.
			
			Now, let we find $x_j$ and $k_j$.
			For (\ref{eqm.1})(\ref{eqm.2}), since $X_0$ is dense in $X$ and the continuity of $T_m$, we can find $x_j\in X$ satisfying $\|x_j\|<\frac{1}{2^j}$ and  $$\sum_{l=1}^{j-1}(x_l+\sum_{i=1}^{N}S_{i,k_l}y_{il})+x_j\in X_0,\quad\bigg\|\frac{1}{n_{k_l}}T_m^{n_{k_l}}x_j\bigg\|<\frac{1}{2^j}\;(l=1,\ldots,j-1). $$
				For (\ref{eqm.7}), by d-Ces$\grave{a}$ro-Hypercyclicity Criterion (1), we can know that there is $k_{j_1}$ so that when $k_j^*>k_{j_1}$, we have
			$$\bigg\|\frac{1}{n_{k_j^*}}T_m^{n_{k_j^*}}\big(\sum_{l=1}^{j-1}(x_l+\sum_{i=1}^{N}S_{i,k_l}y_{il})+x_j\big)\bigg\|<\frac{1}{2^j}.$$
				For (\ref{eqm.5})(\ref{eqm.6}), by d-Ces$\grave{a}$ro-Hypercyclicity Criterion (3), we can know that there is $k_{j_2}$ so that when $k_{j}^*>k_{j_2}$, we have
			$$\bigg\|\frac{1}{n_{k_j^*}}{T_m^{n_{k_j^*}}}{S_{m,k_j^*}}{y_{mj}}-y_{mj}\bigg\|<\frac{1}{2^j}\;\text{and}\;\bigg\|\frac{1}{n_{k_j^*}}{T_m^{n_{k_j^*}}}{S_{i,k_j^*}}{y_{ij}}\bigg\|<\frac{1}{2^j}\;(m\neq i).$$
			For (\ref{eqm.3})(\ref{eqm.4}), by d-Ces$\grave{a}$ro-Hypercyclicity Criterion (2) and the continuity of $T_m$, we can know that there is $k_{j'}>\max\{k_{j_1}+1,k_{j_2}+1,k_{j-1}\}$, we have $$\Big\|S_{i,k_j'}y_{ij}\Big\|<\frac{1}{2^j}\;\text{and}\; \bigg\|\frac{1}{n_{k_l}}{T_m^{n_{k_l}}}{S_{i,k_j'}}{y_{ij}}\bigg\|<\frac{1}{2^j}\;(l=1,\ldots,j-1). $$
			Let $k_j=k_j'$, we find $x_j$ and $k_j$, then (\ref{eqm.1})-(\ref{eqm.7}) hold.
			%$$||x_j||<\frac{1}{2^j},\qquad\bigg|\bigg|\frac{1}{n_{k_l}}{T_m^{n_{k_l}}}{x_j}\bigg|\bigg|<\frac{1}{2^j}\;(l=1,\ldots,j-1).$$
			%$$\Big|\Big|S_{i,k_j}y_{ij}\Big|\Big|<\frac{1}{2^j},\qquad\bigg|\bigg|\frac{1}{n_{k_l}}{T_m^{n_{k_l}}}{S_{i,k_j}}{y_{ij}}\bigg|\bigg|<\frac{1}{2^j}\;(l=1,\ldots,j-1)\;(1\leqslant i\leqslant N ).$$
			%$$\bigg|\bigg|\frac{1}{n_{k_j}}T_m^{n_{k_j}}S_{m,k_j}y_{mj}-y_{mj}\bigg|\bigg|<\frac{1}{2^j},\qquad\bigg|\bigg|\frac{1}{n_{k_j}}T_m^{n_{k_j}}S_{i,k_j}y_{ij}\bigg|\bigg|<\frac{1}{2^j}\;(1\leqslant i\leqslant N,i\neq m).$$
			%$$\bigg|\bigg|\frac{1}{n_{k_j}}T_m^{n_{k_j}}(\sum_{l=1}^{j-1}(x_l+\sum_{i=1}^{N}S_{i,k_l}y_{il})+x_j)\bigg|\bigg|<\frac{1}{2^j}.$$
		
				Since $\|x_j\|<\frac{1}{2^j}$ and $\Big\|S_{i,k_j}y_{ij}\Big\|<\frac{1}{2^j}$, then $x=\sum_{j=1}^{\infty}x_j+\sum_{j=1}^{\infty}\sum_{m=1}^{N}{S_{m,k_j}}{y_{mj}}$ is convergent and for any $y_{lj}\in X$, we have
				\begin{align*}
				\bigg\|\frac{1}{n_{k_j}}T_l^{n_{k_j}} x-y_{lj}\bigg\|&=\bigg\|\frac{1}{n_{k_j}}T_l^{n_{k_j}}\big(\sum_{m=1}^{j-1}(x_m+\sum_{i=1}^{N}S_{i,k_m}y_{im})+x_j\big)+\frac{1}{n_{k_j}}T_l^{n_{k_j}}S_{l,k_j}y_{lj}-y_{lj}+\\
				&\ \ \ \ \frac{1}{n_{k_j}}T_l^{n_{k_j}}(\sum_{i\neq l}S_{i,k_j}y_{ij})+\sum_{m=j+1}^{\infty}\frac{1}{n_{k_j}}T_l^{n_{k_j}}x_m+\sum_{m=j+1}^{\infty}\frac{1}{n_{k_j}}T_l^{n_{k_j}}(\sum_{i=1}^{N}S_{i,k_m}y_{im})\bigg\|\\
				&\leqslant\frac{1}{2^j}+\frac{1}{2^j}+\frac{N-1}{2^j}+\sum_{m=j+1}^{\infty}\frac{1}{2^m}+\sum_{m=j+1}^{\infty}\sum_{i=1}^{N}\bigg\|\frac{1}{n_{k_j}}T_l^{n_{k_j}}S_{i,k_m}y_{im}\bigg\|\\
				&\leqslant\frac{1}{2^j}+\frac{1}{2^j} +\frac{N-1}{2^j}+\sum_{m=j+1}^{\infty}\frac{1}{2^m}+\sum_{m=j+1}^{\infty}\frac{N}{2^m}\\
				&<\frac{2N+3}{2^j}\rightarrow 0\;(j\rightarrow \infty).
				\end{align*}
				%	Thus,$T_1,T_2,\ldots,T_N$ is $disjoint\;Ces\grave{a}ro- hypercyclic operators$.
			Since $\{(y_{1j},y_{2j},\ldots,y_{Nj}):y_{lj}\in X_l,1\leqslant l\leqslant N,j\in \mathbb{N}\}$ is dense in $X^N$, then $x$ is a d-Ces$\grave{a}$ro-hypercyclic vector. Hence  $T_1,T_2,\ldots,T_N$ are d-Ces$\grave{a}$ro-hypercyclic operators.
			\end{proof}

			By constructing a d-Ces$\grave{a}$ro-hypercyclic vector, we have that operators satisfying d-Ces$\grave{a}$ro-Hypercyclicity Criterion are d-Ces$\grave{a}$ro-hypercyclic directly from definition. In the following study, we associate the d-Ces$\grave{a}$ro-Hypercyclicity Criterion with other disjoint dynamical properties. 
			
			%\begin{thm} \label{thm3.7}
			%	Let $T_1,\ldots,T_N$ satisfy the d-Ces$\grave{a}$ro-Hypercyclicity Criterion. Then  $T_1,\ldots,T_N$ are d-Ces$\grave{a}$ro-topologically transitive.% Then $\{T_1^{n_k}\}_{k=1}^{\infty},\ldots,\{T_N^{n_k}\}_{k=1}^{\infty}$ are d-Ces$\grave{a}$ro-mixing. In particaular, $T_1,\ldots,T_N$ are d-Ces$\grave{a}$ro-hypercyclic.
			%\end{thm}
			
		%	\begin{proof}
			%	Let $V_0,V_1,\ldots,V_N$ be open and non-empty subsets of $X$. Pick $y_l\in{V_l\cap X_l}$ and $\epsilon >0$ so that $B(y_l,(N+1)\epsilon)\subset V_l\;(0\leqslant l \leqslant N)$. By assumption, there exists $k_0\in \mathbb{N}$ so that $\frac{1}{n_k}T_l^{n_k}y_0,S_{l,k}y_l$ and $(\frac{1}{n_k}T_l^{n_k}S_{i,k}y_i-\delta_{i,l}y_i)$ belong to $B(0,\epsilon)$ for $k\geqslant k_0$ and $1\leqslant i\leqslant N$. Then for each $k\geqslant k_0$ we have $z=y_0+\sum_{i=1}^{N}S_{i,k}y_i\in V_0$ and $\frac{1}{n_k}T_l^{n_k}z\in  B(y_l,(N+1)\epsilon)\subset V_l\;(1\leqslant l \leqslant N)$. That is, $V_0\cap n_kT_1^{-n_k}(V_1)\cap\ldots \cap n_kT_N^{-n_k}(V_N)\neq\varnothing$ for each $k\geqslant k_0$.
			%\end{proof}
			
			%By proposition \ref{6.2} and theorem \ref{thm3.7}, we can know that operators $T_1,\ldots,T_N$ satisfying the d-Ces$\grave{a}$ro-Hypercyclicity Criterion have a dense set of d-Ces$\grave{a}$ro-hypercyclic vectors.
			%从这里我们可以得到满足准则的算子是不相交拓扑传递的，从而是有稠密的不相交超循环向量。
			\begin{defi}
				Let $(n_k)_k$ be an increasing sequence of positive integers. We say operators $T_1,\ldots,T_N$ are hereditarily d-Ces$\grave{a}$ro-hypercyclic with respect to $(n_k)_k$ if for each subsequence $(n_{k_j})_j$ of $(n_k)_k$, there is some $x\in X$ such that $\big\{(\frac{1}{n_{k_j}}T_1^{n_{k_j}}x,\ldots,\frac{1}{n_{k_j}}T_N^{n_{k_j}}x):j\geqslant 1\big\}$ is dense in $X^N$.
			\end{defi}
			In the following, we give the equivalent conditions of d-Ces$\grave{a}$ro-Hypercyclicity Criterion, which is a special case in \cite{BP}.
			\begin{lem}\label{thm3.9}
				Let $T_l\in B(X)\;(1\leqslant l\leqslant N)$, where $N\geqslant 2$. The following are equivalent:\\
				(a) $T_1,\ldots,T_N$ satisfy the d-Ces$\grave{a}$ro-Hypercyclicity Criterion.\\
				(b) $T_1,\ldots,T_N$ are hereditarily densely d-Ces$\grave{a}$ro-hypercyclic.\\
				(c) For each $r\in\mathbb{N}$, $\overbrace{T_1\oplus\cdots\oplus T_1}^r,\ldots,\overbrace{T_N\oplus\cdots\oplus T_N}^r$ are d-Ces$\grave{a}$ro-topologically transitive on $X^r$.
			\end{lem}
		\begin{thm} \label{th3.8}
			Suppose operators $T_1,\ldots,T_N\in B(X)$ satisfy d-Ces$\grave{a}$ro-Hypercyclicity Criterion, then they are d-Ces$\grave{a}$ro-Blow-Up/Collapse.
			
		\end{thm}
		
		\begin{proof}
			Let $ V_0,V_1,\ldots,V_N $ be non-empty open subsets of $X$ and W be an open neighbourhood of zero. By lemma \ref{thm3.9}, we know that for each $r\in\mathbb{N}$, $\overbrace{T_1\oplus\cdots\oplus T_1}^r,\ldots,\overbrace{T_N\oplus\cdots\oplus T_N}^r$ are d-Ces$\grave{a}$ro-topologically transitive on $X^r$, then there exists $m\in \mathbb{N}$, such that
			$$V_{0,k}\bigcap\cap_{l=1}^{N}mT_l^{-m}(V_{l,k})\neq \varnothing\;(1\leqslant k\leqslant r).$$ For $r=2,V_{0,1}=V_{i,2}=W\;(1\leqslant i\leqslant N),V_{0,2}=V_0,V_{i,1}=V_i\;(1\leqslant i\leqslant N)$,  there exists $m_0\in \mathbb{N}$ such that
			$$W\cap m_0T_1^{-m_0}V_1\cap\cdots\cap m_0T_N^{-m_0}V_N\neq\varnothing,$$
			$$V_0\cap m_0T_1^{-m_0}W\cap\cdots\cap m_0T_N^{-m_0}W\neq\varnothing,$$
			then $T_1,\ldots,T_N$ are d-Ces$\grave{a}$ro-Blow-Up/Collapse.
			%where $r=2,V_{0,1}=V_{i,2}=W\;(1\leqslant i\leqslant N),V_{0,2}=V_0,V_{i,1}=V_i\;(1\leqslant i\leqslant N).$
		\end{proof}
		\begin{rem}
	The following is a summary of the second method to prove that operators satisfying d-Ces$\grave{a}$ro-Hypercyclicity Criterion are d-Ces$\grave{a}$ro-hypercyclic. First, by theorem \ref{th3.8}, we can conclude that operators $T_1, \ldots, T_N$ $( N \geq 2 )$ satisfying the d-Ces$\grave{a}$ro-Hypercyclicity Criterion are d-Ces$\grave{a}$ro-Blow-Up/Collapse. Next, by proposition \ref{pr3.3}, we show that operators satisfying the d-Ces$\grave{a}$ro-Blow-Up/Collapse are d-Ces$\grave{a}$ro-topologically transitive. Then, by proposition \ref{6.2}, we further establish that d-Ces$\grave{a}$ro-topologically transitive operators possess dense d-Ces$\grave{a}$ro-hypercyclic vectors. In summary, operators $T_1, \ldots, T_N$ satisfying d-Ces$\grave{a}$ro-Hypercyclicity Criterion are d-Ces$\grave{a}$ro-hypercyclic. 		
		\end{rem}
		\section{d-ces$\grave{A}$ro-hypercyclic weighted shifts }\label{s4}
		Let $X=c_0(\mathbb{N})$ or $\ell^p(\mathbb{N})(1\leqslant p<\infty)$ over the complex scalar field $\mathbb{C}$. Let $B_a:X\rightarrow X$ be the unilateral weighted backward shift
$$B_{a}(x_0,x_1,x_2,\ldots)= (a_{1}x_1,a_{2}x_2,\ldots),$$
where $a=(a_k)_k$ is a bounded weight sequence. Rolewicz \cite{SR} showed that if $B$ is a unilateral backward shift on the Banach space $\ell^p(\mathbb{N})$, then $\lambda B$ is hypercyclic if and only if $|\lambda| >1$. Salas \cite{SHN} described the characteristics of weight sequences of unilateral weighted backward shifts on $\ell^2(\mathbb{N})$.  Bès and Peris \cite{BP} described the characteristics of weight sequences of disjoint hypercyclic unilateral weighted backward shifts on $\ell^p(\mathbb{N})\;(1\leqslant p<\infty)$ or $c_0(\mathbb{N})$. León-Saavedra \cite{LS} described the characteristics of weight sequences of Ces$\grave{a}$ro-hypercyclic unilateral weighted backward shifts on $\ell^2(\mathbb{N})$ and illustrated every Ces$\grave{a}$ro-hypercyclic unilateral weighted backward shift is hypercyclic.
		
	Next, we describe the characteristics of weight sequence of disjoint Ces$\grave{a}$ro-hypercyclic unilateral weighted backward shifts.		
		\begin{thm}\label{4.1}
		
		Let $X=c_0(\mathbb{N})$ or $\ell^2(\mathbb{N})$ over the complex scalar field $\mathbb{C}$ and let integers $1\leqslant r_1< r_2<\cdots < r_N$ be given. For each $1\leqslant l\leqslant N$, let $a_l=(a_{l,n})_{n=1}^\infty$ be a bounded weight sequence and $B_{a_l}:X\rightarrow X$ be the corresponding unilateral weighted backward shift
		$$B_{a_l}(x_0,x_1,x_2,\ldots)= (a_{l,1}x_1,a_{l,2}x_2,\ldots).$$
		Then the following are equivalent:\\
		(a) $B_{a_1}^{r_1},\ldots,B_{a_N}^{r_N}$ are d-Ces$\grave{a}$ro-hypercyclic.\\
(b) For each $\epsilon>0$ and $q\in\mathbb{N}$, there exists $m\in\mathbb{N}$ satisfying: for each $0\leqslant j\leqslant q$,
		$$\bigg|\frac{1}{m}a_{l,j+1}\ldots a_{l,j+r_{l}m}\bigg|>\frac{1}{\epsilon}\qquad(1\leqslant l\leqslant N),$$
		$$\Bigg|\frac{a_{l,j+1}\ldots a_{l,j+r_{l}m}}{a_{s,j+(r_l-r_s)m+1}\ldots a_{s,j+r_lm}}\bigg|>\frac{1}{\epsilon}\qquad(1\leqslant s<l\leqslant N).$$
		(c) $B_{a_1}^{r_1},\ldots,B_{a_N}^{r_N}$ satisfy the d-Ces$\grave{a}$ro-Hypercyclicity Criterion.
		
		\end{thm}
		\begin{proof}
			
			(a)$\Rightarrow$(b).
 			Let $\epsilon>0$ and $q\in\mathbb{N}$ be given, pick $0<\delta<1$ with $\delta/(1-\delta)<\epsilon$, and let $x=(x_0,x_1,x_2,\ldots)$ be a d-Ces$\grave{a}$ro-hypercyclic vector for $B_{a_1}^{r_1},\ldots,B_{a_N}^{r_N}$. Now, let $m\in\mathbb{N}\;(m>q)$ so that for $k\geqslant r_1m$,
			$$|x_k|<\delta,$$
			and
			$$\bigg\|\frac{1}{m}B_{a_l}^{r_lm}x-(e_0+e_1+\cdots+e_q)\bigg\|<\delta\qquad(1\leqslant l\leqslant N).$$
			Then for $1\leqslant l\leqslant N$, we have
			\begin{equation*}
				\begin{cases}
			1-\delta<\Big|\frac{1}{m}a_{l,i+1}\ldots a_{l,i+r_lm}x_{i+r_lm}\Big|<1+\delta,\quad&\text{if}\;0\leqslant i\leqslant q,\\
			\Big|\frac{1}{m}a_{l,i+1}\ldots a_{l,i+r_lm}x_{i+r_lm}\Big|<\delta,\quad&\text{if}\; i>q.
		\end{cases}
	\end{equation*}	
	
			Now, let $0\leqslant j\leqslant q$ and $1\leqslant l\leqslant N$ be fixed, we have
			$$\bigg|\frac{1}{m}a_{l,j+1}\ldots a_{l,j+r_lm}\bigg|>\frac{1-\delta}{|x_{j+r_lm}|}>\frac{1-\delta}{\delta}>\frac{1}{\epsilon}.$$
			
			Also, for $1\leqslant s<l\leqslant N$,
			\begin{equation*}
				\begin{split}
				\frac{\Big|a_{l,j+1}\ldots a_{l,j+r_lm}\Big|}{\Big|a_{s,j+(r_l-r_s)m+1}\ldots a_{s,j+r_lm}\Big|}&=
				\frac{\Big|\displaystyle{\frac{a_{l,j+1}\ldots a_{l,j+r_lm}x_{j+r_lm}}{m}}\Big|}{\Big|\displaystyle{\frac{a_{s,j+(r_l-r_s)m+1}\ldots a_{s,j+r_lm}x_{j+r_lm}}{m}}\Big|}\\
				&>\frac{1-\delta}{\Big|\displaystyle{\frac{a_{s,j+(r_l-r_s)m+1}\ldots a_{s,j+r_lm}x_{j+r_lm}}{m}}\Big|}\\
				&=\frac{1-\delta}{\Big|\displaystyle{\frac{a_{s,i+1}\ldots a_{s,i+r_sm}x_{i+r_sm}}{m}}\Big|}\\
				&> \frac{1-\delta}{\delta}\\
				&>\frac{1}{\epsilon},
				\end{split}
			\end{equation*}
			where $i=j+(r_l-r_s)m$, $i>q$.
			
			(b)$\Rightarrow$(c). By (b), there exist integers $1\leqslant n_1<n_2<\cdots$ satisfying for each $q\in \mathbb{N}$ and each $0\leqslant j\leqslant q$,
			\begin{equation}\label{eq4.1}
				\Big|\frac{a_{l,j+1}\ldots a_{l,j+r_ln_q}}{n_q}\Big|>q\qquad(1\leqslant l\leqslant N),
				\end{equation}
				\begin{equation}\label{eq4.2}
					\frac{\Big|a_{l,j+1}\ldots a_{l,j+r_ln_q}\Big|}{\Big|a_{s,j+(r_l-r_s)n_q+1}\ldots a_{s,j+r_ln_q}\Big|}>q\qquad(1\leqslant s <l\leqslant N).
					\end{equation}	
			%$$\Big|\frac{a_{l,j+1}\ldots a_{l,j+r_ln_q}}{n_q}\Big|>q\qquad(1\leqslant l\leqslant N),$$
			%$$\frac{\Big|\frac{a_{l,j+1}\ldots a_{l,j+r_ln_q}}{n_q}\Big|}{\Big|\frac{a_{s,j+(r_l-r_s)m+1}\ldots a_{s,j+r_ln_q}}{n_q}\Big|}>q\qquad(1\leqslant s <l\leqslant N).$$
			
			 Let $X_0=\cdots=X_N=\text{span}\{e_0,e_1,\ldots\}$, then $X_0$ is dense in $X$ and we have $\frac{1}{n_q}B_{a_l}^{r_ln_q}\xrightarrow{q\rightarrow \infty} 0 $ pointwise on $X_0\;(1\leqslant l\leqslant N)$. 
			
			For each $1\leqslant l\leqslant N$ and $q\in\mathbb{N}$, consider the mapping $S_{l,q}:X_0\rightarrow X$ given by
			$$S_{l,q}(x_0,x_1,\ldots)=nq\Big(\overbrace{0,\ldots,0}^{r_ln_q},\frac{x_0}{a_{l,1}\ldots a_{l,r_ln_q}},\frac{x_1}{a_{l,2}\ldots a_{l,r_ln_q+1}},\ldots\Big),$$
			then by (\ref{eq4.1}) $S_{l,q}\xrightarrow{q\rightarrow \infty} 0$ and we have $\frac{1}{n_q}B_{a_l}^{r_ln_q}S_{l,q}-Id_{X_0}\xrightarrow{q\rightarrow \infty} 0$ pointwise on $X_0$.
			%So $\frac{1}{n_q}B_{a_l}^{r_ln_q}S_{l,q}=Id_{X_0}$ and $S_{l,q}\rightarrow 0$ pointwise on $X_0\;(1\leqslant l\leqslant N)$ as $q\rightarrow \infty$.
			
			Now, for $1\leqslant s< l\leqslant N$, we have
			$$\frac{1}{n_q}B_{a_l}^{r_ln_q}S_{s,q}\xrightarrow{q\rightarrow \infty} 0 \;\text{pointwise on}\; X_0,\;\text{since}\;r_s<r_l.$$
			Also,
			$$\frac{1}{n_q}B_{a_s}^{r_sn_q}S_{l,q}(x_0,x_1,\ldots)=(0,\ldots,0,\frac{a_{s,r_ln_q}\ldots a_{s,(r_l-r_s)n_q+1}x_0}{a_{l,1}\ldots a_{l,r_ln_q}},\ldots)$$
			then by (\ref{eq4.2}) $\frac{1}{n_q}B_{a_s}^{r_sn_q}S_{l,q}\xrightarrow{q\rightarrow \infty} 0$ pointwise on $X_0$. So $B_{a_1}^{r_1},\ldots,B_{a_N}^{r_N}$ satisfy the d-Ces$\grave{a}$ro-Hypercyclicity Criterion.
			
			(c)$\Rightarrow$(a). Obivously.
		\end{proof}
		
		\begin{cor}
			 Let $X= \ell^2(\mathbb{N})$, then disjoint Ces$\grave{a}$ro-hypercyclic unilateral weighted backward shifts on $X$ are disjoint hypercyclic.
				\end{cor}
		\begin{ex}
			For $(N\geqslant 2)$, Let $B_{a_1}^{r_1},\ldots,B_{a_N}^{r_N}\in \ell^2(\mathbb{N})$ are unilateral weighted backward shifts, where $a_l=(a_{l,i})_{i=1}^{\infty}=l+1,r_l=l\;(1\leqslant l\leqslant N)$, then they are disjoint Ces$\grave{a}$ro-hypercyclic unilateral weighted backward shifts.
			
		\end{ex}
		
		%a_0=,AAa_i=?
		\begin{proof}
		If $1\leqslant l\leqslant N$, then
		$$	\lim_{n\rightarrow\infty}\Big|\frac{1}{n}\textstyle\prod_{i=j+1}^{j+r_ln}a_{l,i}\Big|=\lim\limits_{n\rightarrow\infty}\Bigg|\displaystyle{\frac{(l+1)^{r_ln}}{n}}\Bigg|=\infty\;(1\leqslant l\leqslant N),$$ 
		and if $1\leqslant s<l\leqslant N$, then
		$$\lim_{n\rightarrow\infty}\Bigg|\frac{\prod_{i=j+(r_l-r_s)n+1}^{j+r_ln}a_{s,i}}{\prod_{i=j+1}^{j+r_ln}a_{l,i}}\Bigg|=\lim_{n\rightarrow\infty}\Bigg|\frac{(s+1)^{r_sn}}{(l+1)^{r_ln}}\Bigg|=0.$$
		Hence, by theorem \ref{4.1}, $B_{a_1}^{r_1},\ldots,B_{a_N}^{r_N}$ are disjoint Ces$\grave{a}$ro-hypercyclic.
			\end{proof}
			
			Let $B_a:X\rightarrow X$ be the bilateral weighted shift, $(e_k)_k$ be a basis in $X$ and $a=(a_k)_k$ be a bounded weight sequence
			$$B_ae_k=a_ke_{k-1}\;(k\in\mathbb{Z}).$$
			
	Salas \cite{SHN} described the characteristics of weight sequences of bilateral weighted shifts on $\ell^2(\mathbb{Z})$. Bès and Peris \cite{BP} described the characteristics of weight sequences of disjoint hypercyclic bilateral weighted shifts on $c_0(\mathbb{Z})$ or $\ell^p(\mathbb{Z})\;(1\leqslant p<\infty)$. León-Saavedra \cite{LS} described the characteristics of weight sequences of Ces$\grave{a}$ro-hypercyclic bilateral weighted shifts on $\ell^2(\mathbb{Z})$.	
			
	Next, we describe the characteristics of weight sequence of disjoint Ces$\grave{a}$ro-hypercyclic bilateral weighted shifts.		
		\begin{thm}
			Let $X=c_0(\mathbb{Z})$ or $\ell^2(\mathbb{Z})$ over the complex scalar field $\mathbb{C}$ and integers $1\leqslant r_1< r_2<\cdots < r_N$ be given. For each $1\leqslant l\leqslant N$, let $a_l=(a_{l,n})_{n\in\mathbb{Z}}$ be a bounded bilateral weight sequence and $B_{a_l}:X\rightarrow X$ be the corresponding bilateral weighted shift $B_{a_l}e_k=a_{l,k}e_{k-1}(k\in\mathbb{Z})$.\\
			%Let $X=c_0(\mathbb{Z})$ or $l^p(\mathbb{Z})\;(1\leqslant p <\infty)$. For $l=1,\ldots,N$, let $a_l=(a_{l,j})_{j\in\mathbb{Z}}$ be a bounded bilateral sequence of non-zero scalars,and let $B_{a_l}$ be the associated backward shift on $X$ given by $B_{a_l}e_k=a_{l,k}e_{k-1}(k\in\mathbb{Z})$. For any integers $1\leqslant r_1<\ldots r_N$,the following are equivalent:
			Then the following are equivalent:\\
			(a) The set of $d-CH(B_{a_1}^{r_1},\ldots,B_{a_N}^{r_N})$ is a dense set.\\
			(b) For each $\epsilon>0$ and $q\in\mathbb{N}$, there exists $m\in\mathbb{N}$ satisfying: for each $| j|\leqslant q$, we have:\\
			If $ 1\leqslant l\leqslant N$,
			\begin{equation*}
				\begin{cases}
					\Big|\frac{\prod_{i=j+1}^{j+r_lm}a_{l,i}}{m}\Big|>\frac{1}{\epsilon},\\
					\Big|\frac{\prod_{i=j-r_lm+1}^{j}a_{l,i}}{m}\Big|<\epsilon.
				\end{cases}
			\end{equation*}
			If $ 1\leqslant s<l\leqslant N$,
			\begin{equation*}
				\begin{cases}
					\Big|\prod_{i=j+1}^{j+r_lm}a_{l,i}\Big|>\frac{1}{\epsilon}\Big|\prod_{i=j+(r_l-r_s)m+1}^{j+r_lm}a_{s,i}\Big|,\\
					\Big|\prod_{i=j+(r_s-r_l)m+1}^{j+r_sm}a_{l,i}\Big|<\epsilon\Big|\prod_{i=j+1}^{j+r_sm}a_{s,i}\Big|.
				\end{cases}
			\end{equation*}		
	(c)$B_{a_1}^{r_1},\ldots,B_{a_N}^{r_N}$ satisfy the d-Ces$\grave{a}$ro-Hypercyclicity Criterion.
		\end{thm}
		
		\begin{proof}
			(a)$\Rightarrow$(b). Let $\epsilon >0$ and $q\in\mathbb{N}$ be given. Pick $0<\delta <1$ so that ${\delta}/(1-\delta)<\epsilon$. Let $x=(x_j)_{j\in\mathbb{Z}}$ be a d-Ces$\grave{a}$ro-hypercyclic vector satisfying
			$$\bigg\|x-\sum_{|j|\leqslant q}e_j\bigg\|<\delta.$$
			Let $m\in\mathbb{N}\;(m>2q)$ so that for $|k|\geqslant r_1m$,
			$$|x_k|<\delta,$$
			and
			$$\bigg\|\frac{1}{m}{B_{a_l}^{r_lm}}{x}-\sum_{|j|\leqslant q}e_j\bigg\|<\delta.$$
			Then we have 
			\begin{equation*}
				\begin{cases}
					|x_j-1|<\delta,\;&\text{if}\;|j|\leqslant q,\\
					|x_j|<\delta,\;&\text{if}\;|j|>q.
				\end{cases}
			\end{equation*}	
			And 
				\begin{equation*}
				\begin{cases}
					\bigg|\frac{1}{m}(\prod_{i=j+1}^{j+r_lm}a_{l,i})x_{j+r_lm}-1\bigg|<\delta,\quad&\text{if}\;|j|\leqslant q,\\
					\bigg|\frac{1}{m}(\prod_{i=j+1}^{j+r_lm}a_{l,i})x_{j+r_lm}\bigg|<\delta,\quad&\text{if}\;|j|>q.
						\end{cases}
				\end{equation*}	
		%	$$\bigg|\frac{1}{m}(\prod_{i=j+1}^{j+r_lm}a_{l,i})x_{j+r_lm}-1\bigg|<\delta\quad\text{if}\;|j|\leqslant q,$$
		%	$$\bigg|\frac{1}{m}(\prod_{i=j+1}^{j+r_lm}a_{l,i})x_{j+r_lm}\bigg|<\delta\quad\text{if}\;|j|>q.$$
		
			Let $|j|\leqslant q $ be fixed and for $1\leqslant l\leqslant N$, then we have
			$$\bigg|\frac{1}{m}\textstyle(\prod_{i=j+1}^{j+r_lm}a_{l,i})\bigg|>\displaystyle{\frac{1-\delta}{|x_{j+r_lm}|}}>\displaystyle{\frac{1-\delta}{\delta}}>\displaystyle{\frac{1}{\epsilon}},$$
				\begin{equation*}
					\begin{split}
			\bigg|\frac{1}{m}\textstyle(\prod_{i=j-r_lm+1}^{j}a_{l,i})\bigg| &=\bigg|\frac{1}{m}\textstyle(\prod_{i=k+1}^{k+r_lm}a_{l,i})\bigg| \\
			&=\bigg|\frac{\frac{1}{m}\prod_{i=k+1}^{k+r_lm}a_{l,i}x_{k+r_lm}}{x_{k+r_lm}}\bigg| \\
			&<\frac{\delta}{|x_{k+r_lm}|} \\
			&=\frac{\delta}{|x_j|}\\
			&<\frac{\delta}{1-\delta}\\
			&<\epsilon,
			\end{split}
				\end{equation*}
				where $k=j-r_lm$, $|k|>q$.
				
			Also, for $1\leqslant s<l\leqslant N$ and
			$|j|\leqslant q,$
			\begin{equation*}
			\begin{split}
			\bigg|\frac{(\prod_{i=j+1}^{j+r_lm}a_{l,i})}{(\prod_{i=j+(r_l-r_s)m+1}^{j+r_lm}a_{s,i})}\bigg|&=\bigg|\frac{\frac{1}{m}(\prod_{i=j+1}^{j+r_lm}a_{l,i})x_{j+r_lm}}{\frac{1}{m}(\prod_{i=j+(r_l-r_s)m+1}^{j+r_lm}a_{s,i})x_{j+r_lm}}\bigg|\\
			&=\bigg|\frac{\frac{1}{m}(\prod_{i=j+1}^{j+r_lm}a_{l,i})x_{j+{r_lm}}}{\frac{1}{m}(\prod_{i=k+1}^{k+r_sm}a_{s,i})x_{k+{r_sm}}}\bigg|\\
			&>\frac{1-\delta}{\delta}\\
			&>\frac{1}{\epsilon},
		\end{split}
			\end{equation*}
		where $k=j+(r_l-r_s)m$, $|k|>q$.
			\begin{equation*}
				\begin{split}
			\bigg|\frac{(\prod_{i=j+(r_s-r_l)m+1}^{j+r_sm}a_{l,i})}{(\prod_{i=j+1}^{j+r_sm}a_{s,i})}\bigg|&=\bigg|\frac{\frac{1}{m}(\prod_{i=j+(r_s-r_l)m+1}^{j+r_sm}a_{l,i})x_{j+r_sm}}{\frac{1}{m}(\prod_{i=j+1}^{j+r_sm}a_{s,i})x_{j+r_sm}}\bigg|\\
			&<\bigg|\frac{\frac{1}{m}(\prod_{i=j+(r_s-r_l)m+1}^{j+r_sm}a_{l,i})x_{j+r_sm}}{1-\delta}\bigg|\\
			&=\bigg|\frac{\frac{1}{m}(\prod_{i=k+1}^{k+r_lm}a_{l,i})x_{k+{r_lm}}}{1-\delta}\bigg|\\
			&<\frac{\delta}{1-\delta}\\
			&<\epsilon,
			\end{split}
		\end{equation*}
			where $k=j+(r_s-r_l)m$, $|k|>q$.
			
			(b)$\Rightarrow$(c). By(b), we may get integers $1\leqslant n_1<n_2<\cdots$ so that for $|j|\leqslant q$ we have:
			for each $1\leqslant l\leqslant N$,
			\begin{equation}\label{eq4.3}
				\begin{cases}
					\Big|\frac{1}{n_q}(\prod_{i=j+1}^{j+r_ln_q}a_{l,i})\Big|>q,\\
					\Big|\frac{1}{n_q}(\prod_{i=j-r_ln_q+1}^{j}a_{l,i})\Big|<\frac{1}{q},
				\end{cases}
			\end{equation}
			and for $1\leqslant s<l\leqslant N$,
			\begin{equation}\label{eq4.4}
				\begin{cases}
					\Big|(\prod_{i=j+1}^{j+r_ln_q}a_{l,i})\Big|>q\Big|(\prod_{i=j+(r_l-r_s)n_q+1}^{j+r_ln_q}a_{s,i})\Big|,\\
					\Big|(\prod_{i=j+(r_s-r_l)n_q+1}^{j+r_sn_q}a_{l,i})\Big|<\frac{1}{q}\Big|(\prod_{i=j+1}^{j+r_sn_q}a_{s,i})\Big|.
				\end{cases}
			\end{equation}	
			
			Now, let $X_0=\cdots=X_N=\text{span}\{e_k\}_{k\in\mathbb{Z}}$, then $X_0$ is dense in $X$. For $1\leqslant l\leqslant N$, 
			$$\frac{1}{n_q}B_{a_l}^{r_ln_q}e_k=\frac{a_{l,k}\ldots a_{l,k-r_ln_q+1}}{n_q}e_{k-r_ln_q},$$
			by (\ref{eq4.3}), we have
			  $\frac{1}{n_q}B_{a_l}^{r_ln_q}\xrightarrow{q\rightarrow \infty} 0 $ pointwise on $X_0\;(1\leqslant l\leqslant N)$. 
			  
			  For each $(1\leqslant l\leqslant N)$ and $q\in\mathbb{N}$, let the mapping $S_{l,q}                                                                                                                                                                                                                                                                                                                                                                                                                                                                                                                                                                                                                                                                                                                                                                                                                                                                                                                                                                                                                                                                                                                                                                                                                                                                                                                                                                                                                                                                                                                                                                                                                                                                                                                                                                                                                                                                                                                                                                                                                                                                                                                                                                                                                                                                                                                                                                                                                                                                                                                                                                                                                                                                                                                                                                                                                                                                                                                                                                                                                                                                                                                                                                                                                                                                                                                                                                                                                                                                                                                                                                                                                                                                                                                                                                                                                                                                                                                                                                                                                                                                                                                                                                                                                                                                                                                                                                                                                                                                                                                                                                                                                                                                                                                                                                                                                                                                                                                                                                                                                                                                                                                                                                                                                                                                                                                                                                                                                                                                                                                                                                                                                                                                                                                                                                                                                                                                                                                                                                                                                                                                                                                                                                                                                                                                                                                                                                                                                                                                                                                                                                                                                                                                                                                                                                                                                                                                                                                                                                                                                                                                                                                                                                                                                                                                                                                                                                                                                                                                                                                                                                                                                                                                                                                                                                                                                                                                                                                                                                                                                                                                                                                                                                                                                                                                                                                                                                                                                                                                                                                                                                                                                                                                                                                                                                                                                                                                                                                                                                                                                                                                                                                                                                                                                                                                                                                                                                                                                                                                                                                                                                                                                                                                                                                                                                                                                                                                                                                                                                                                                                                                                                                                                                                                                                                                                                                                                                                                                                                                                                                                                                                                                                                                                                                                                                                                                                                                                                                                                                                                                                                                                                                                                                                                                                                                                                                                                                                                                                                                                                                                                                                                                                                                                                                                                                                                                                                                                                                                                                                                                                                                                                                                                                                                                                                                                                                                                                                                                                                                                                                                                                                                                                                                                                                                                                                                                                                                                                                                                                                                                                                                                                                                                                                                                                                                                                                                                                                                                                                                                                                                                                                                                                                                                                                                                                                                                                                                                                                                                                                                                                                                                                                                                                                                                                                                                                                                                                                                                                                                                                                                                                                                                                                                                                                                                                                                                                                                                                                                                                                                                                                                                                                                                                                                                                                                                                                                                                                                                                                                                                                                                                                                                                                                                                                                                                                                                                                                                                                                                                                                                                                                                                                                                                                                                                                                                                                                                                                                                                                                                                                                                                                                                                                                                                                                                                                                                                                                                                                                                                                                                                                                                                                                                                                                                                                                                                                                                                                                                                                                                                                                                                                                                                                                                                                                                                                                                                                                                                                                                                                                                                                                                                                                                                                                                                                                                                                                                                                                                                                                                                                                                                                                                                                                                                                                                                                                                                                                                                                                                                                                                                                                                                                                                                                                                                                                                                                                                                                                                                                                                                                                                                                                                                                                                                                                                                                                                                                                                                                                                                                                                                                                                                                                                                                                                                                                                                                                                                                                                                                                                                                                                                                                                                                                                                                                                                                                                                                                                                                                                                                                                                                                                                                                                                                                                                                                                                                     :X_0\rightarrow X$ be the linear map given by
			$$S_{l,q}(e_k)=\frac{n_q}{a_{l,k+1}\ldots a_{l,k+r_ln_q}}e_{k+r_ln_q}\;(k\in\mathbb{Z}),$$
			so  $\frac{1}{n_q}B_{a_l}^{r_ln_q}S_{l,q}=Id_{X_0}$ and by (\ref{eq4.3}) $S_{l,q}\xrightarrow{q\rightarrow \infty} 0$ pointwise on $X_0\;(1\leqslant l\leqslant N)$. 
			
			Also, for $1\leqslant s< l\leqslant N$, we have
			$$\frac{1}{n_q}B_{a_l}^{r_ln_q}S_{s,q}e_k=\frac{\prod_{i=k+(r_s-r_l)n_q+1}^{k+r_sn_q}a_{l,i}}{\prod_{i=k+1}^{k+r_sn_q}a_{s,i}}e_{k+(r_s-r_l)n_q}\;(k\in\mathbb{Z}),$$
			then by (\ref{eq4.4}) $\frac{1}{n_q}B_{a_l}^{r_ln_q}S_{s,q}\xrightarrow{q\rightarrow\infty} 0$ pointwise on $X_0.$

			For $1\leqslant s<l\leqslant N$,
			$$\frac{1}{n_q}B_{a_s}^{r_sn_q}S_{l,q}e_k=\frac{\prod_{i=k+(r_l-r_s)n_q+1}^{k+r_ln_q}a_{s,i}}{\prod_{i=k+1}^{k+r_ln_q}a_{l,i}}e_{k+(r_l-r_s)n_q}\;(k\in\mathbb{Z}),$$
			then by (\ref{eq4.4}) $\frac{1}{nq}B_{a_s}^{r_sn_q}S_{l,q}\xrightarrow{q\rightarrow\infty}0$ pointwise on $X_0.$ So $B_{a_1}^{r_1},\ldots,B_{a_N}^{r_N}$ satisfy the d-Ces$\grave{a}$ro-Hypercyclicity Criterion.
			
			(c)$\Rightarrow$(a). Obviously.
		\end{proof}

\end{document}